\documentclass[12pt]{amsart}
\usepackage[top=1.5in, bottom=1.5in, left=1.25in, right=1.25in]	{geometry}

\usepackage{mathtools}
\mathtoolsset{showonlyrefs,showmanualtags}

\usepackage{hyperref} 
\hypersetup{
    colorlinks=true,       
    linkcolor=blue,          
    citecolor=magenta,        
    filecolor=magenta,      
    urlcolor=cyan           
}

\usepackage[backrefs]{amsrefs}

\usepackage[all]{xy}
\usepackage{amsmath}
\usepackage{amssymb}
\usepackage{amsthm}
\usepackage{amscd}

\newcommand{\R}{\mathcal{R}}

\newcommand{\F}{\mathcal{F}}

\newcommand{\Z}{\Bbb Z}

\newcommand{\RR}{\Bbb R}
\newcommand{\TT}{\mathbb{ T} }

\newcommand{\Ca}{\mathcal{C}}

\newtheorem{conjecture}[equation]{Conjecture}
\newtheorem{theorem}[equation]{Theorem}
\newtheorem{definition}[equation]{Definition}

\newtheorem{cor}[equation]{Corollary}
\newtheorem{lemma}[equation]{Lemma}

\numberwithin{equation}{section}

\usepackage{mathtools}
\mathtoolsset{showonlyrefs,showmanualtags} 

\usepackage{hyperref} 
\hypersetup{
    colorlinks=true,       
    linkcolor=blue,          
    citecolor=magenta,        
    filecolor=magenta,      
    urlcolor=cyan           
}

\begin{document}
\title[Discrete Quadratic Carleson Theorem]{A  Discrete Quadratic Carleson Theorem on $ \ell ^2 $ with 
a Restricted Supremum} 
\author{Ben Krause}
\address{
Department of Mathematics
The University of British Columbia \\
1984 Mathematics Road
Vancouver, B.C.
Canada V6T 1Z2}
\email{benkrause@math.ubc.ca}
\thanks{Research supported in part by a NSF Postdoctoral Research Fellowship.}
\author{Michael T. Lacey}
\address{
School of Mathematics
Georgia Institute of Technology \\
686 Cherry Street
Atlanta, GA 30332-0160}
\email{lacey@math.gatech.edu}
\thanks{Research supported in part by grant NSF-DMS 1265570 and  NSF-DMS-1600693.} 
\date{\today}
\maketitle

\begin{abstract}
Consider the discrete maximal function acting on $ \ell ^2(\Z)$ functions
\[ \mathcal{C}_{\Lambda} f( n ) := \sup_{ \lambda \in \Lambda} \biggl\lvert  \sum_{m \neq 0} f(n-m) \frac{ e ^{2 \pi i\lambda m^2}}  {m} \biggr\rvert \]
where $\Lambda \subset [0,1]$. We give sufficient conditions on $ \Lambda $, met by  certain kinds of Cantor sets, for this to be a bounded sublinear operator. 
This result is  a discrete analogue of E.~M.~Stein's  integral result, that the maximal operator 
below is bounded on $ L ^2 (\mathbb R )$.  
\[ \mathcal{C}_2 f(x):= \sup_{\lambda \in \RR} \left | \int f(x-y) \frac{e^{2\pi i \lambda y^2}}{y} \ dy \right|.\]
The proof of our result relies heavily on Bourgain's work on arithmetic ergodic theorems, with novel 
complexity arising from the oscillatory nature of the question at hand, and  difficulties arising  from the 
 the parameter $ \lambda $ above.  
\end{abstract}

 \setcounter{tocdepth}{1}
\tableofcontents

\section{Introduction}
An important theme of harmonic analysis, widely promoted by Elias M.~Stein over the last half-century, explores the role of 
oscillatory integrals, in themselves, or as operators.  Among this rich topic, the key for this paper is the following 
maximal quadratic operator defined by the formula below,
\begin{equation}\label{e:stein}
\mathcal{C}_{2}f(x) := \sup_{\lambda \in \RR} \left | \int f(x-y) \frac{e(\lambda y^2)}{y} \ dy \right|
\end{equation}  
Here and throughout, 
$
e (t) = e ^{2 \pi i t}
$.  
Stein \cite{MR1364908}  showed that $ \mathcal{C}_{2}$ is bounded on $ L ^{p}$, for $ 1< p < \infty $.    
This result  has had some profound extensions,  encompassing Carleson's Theorem \cite{MR0199631} on the convergence of 
Fourier series in important work of Victor Lie \cite{MR2545246,11054504}, and more recently to work of Pierce-Yung \cite{150503882}, on related operators in the  Radon transform setting.

Our main result is an initial result in the discrete setting, that is replacing integrals by sums, in which case 
certain arithmetic considerations are paramount.  
The polynomial ergodic theorems of Bourgain \cite{MR937581,MR937582} are the seminal work in this direction. 
Bourgain's and subsequent work show that the  discrete inequalities require novel insights and 
technique to establish, and so are profound extensions of their integral counterparts.    
 Among a substantial literature that has flowed from Bourgain's work, we point to 
 a few older papers 
 \cites{MR2318564,MR1056560,MR1719802},  more recent papers 
 \cites{14021803,MR3370012,MR2872554,MR3280058,MR2188130}, and very recent papers \cites{151207518,151207523,151207524}.

Lillian Pierce informed us of this conjecture, see the acknowledgments, which in a certain sense unifies  
the direction of Stein and that of Bourgain. 

\begin{conjecture}\label{j:}  The following inequality holds on $ \ell ^2 (\mathbb Z )$. 
\begin{equation*}
\Bigl\lVert  \sup _{0 \leq  \lambda \leq 1 } 
\Bigl\lvert  \sum_{n\neq 0}  f (x-n)  \frac {e (\lambda  n ^2 )} n  \Bigr\rvert\, 
\Bigr\rVert _{\ell ^2 } \lesssim \lVert f\rVert _{\ell  ^2 }. 
\end{equation*}
\end{conjecture}
Note that we can assume that our modulation parameters live inside a countable set. We will implicitly use this assumption throughout the paper.

We provide supporting evidence for Pierce's conjecture, by  
 exhibiting a class of infinite sets  $ \Lambda \subset [0,1]$ for which restricting the supremum to $ \Lambda $ 
\[ \mathcal{C}_{\Lambda} f( n ) := \sup_{ \lambda \in \Lambda} \left| \sum_{m \neq 0} f(n-m) \frac{ e(\lambda m^2)}{m} \right| \]
the maximal function is $\ell^2(\mathbb Z)$ bounded.
The type of condition that we impose on $ \Lambda $ is described here.  

\begin{definition}\label{d:amk} A set $ \Lambda \subset [0,1]$ has \emph{arithmetic Minkowski dimension $ d$},  
where $ 0< d \leq 1$ if there is a constant $ C _{\Lambda }$ so that for all $ 0<t<1$, the set $ \Lambda $ can be covered by intervals 
$ I_ 1 ,\dotsc, I_N$, where each $ I_n$ has length at most $ t$, and is 
 centered at a  rational $ \frac aq$, where $1\leq  q \leq  C _{\Lambda }t ^{-d}$  
(and so $ N\leq   C _{\Lambda } t ^{2d}$). 
\end{definition}

Note that the condition above implies that $ \Lambda $ has upper Minkowski dimension $ 2d$.  
There  are non-trivial examples of such sets.  For integer   $ D \geq 2$, define a  
 Cantor-like set  $ \Lambda $ below. 
\begin{equation*}
\Lambda := \biggl\{  \sum_{  j \in J }  2^{-D^{j} } \;:\; J\subset \mathbb N  \biggr\}. 
\end{equation*}
We can cover this set with intervals of length $ t= 2 ^{1-D ^{n+1}}$,  
with each interval centered at a rational with denominator $ 2 ^{D ^{n}} \leq  2t ^{- 1/D}$. 
Hence, the arithmetic Minkowski 
dimension of $ \Lambda $ is at most $ 1/D$.

Under the above condition, we prove the following theorem.

\begin{theorem}\label{t:main}
There is a $ 0< d < 1$  so that for any $ \Lambda \subset [0,1]$ of arithmetic Minkowski dimension $ d$, the 
inequality below holds. 
\[ \| \Ca_{\Lambda} f \|_{\ell^2(\mathbb Z)}  \leq C _{\Lambda }  \|f\|_{\ell^2(\mathbb Z)}.\]
\end{theorem}

Concerning the proof, in continuous variants of these questions, the method of $ T  T ^{\ast} $ has been dominant, in the case of a fixed polynomial \cite{MR822187}, and in the maximal variants \cites{MR1364908,MR1719802}.  But this technique  in the discrete analogs \cites{MR922412,MR1019960,MR1719802}  would require sophisticated  version of Weyl's Lemma, to address the `minor arcs'.  But, this  Lemma is not available, so alternate approach is required.  

Following Bourgain's lead, we treat the Theorem above as a maximal multiplier theorem, where the multipliers are a function of $ \lambda $, given by 
\begin{equation}\label{e:Mlb}
M(\lambda,\beta) =  \sum_ {m \neq 0} \frac{e(\lambda m^2 -\beta m)}{m} . 
\end{equation}
A detailed description of these multipliers is available from the Hardy-Littlewood circle method, an analysis 
that is carried out in \S \ref{s:circle}. 
As is typical, the analysis splits into the `major and minor arcs.'   

The minor arcs component has little structure, except that the multiplier is `small.'   (Weyl's Lemma is relevant here.)
In the $ \ell ^2 $-case of Bourgain's ergodic theorems, this information and a trivial 
square function argument dispenses with the minor arcs. 
This is not the case in the present study, as the minor arcs arise for each choice of $ \lambda  $.  It is 
hardly clear how they are related as $ \lambda $ varies, except through a trivial derivative estimate. 
Addressing this is the first crucial way that the dimensional hypothesis is used, in combination with 
an elementary device described in Lemma~\ref{SOB}.

That leaves the major arcs, centered at rationals with small denominator. 
Around them, as is usual, the discrete sum $ M (\lambda , \beta )$  is approximated by a multiplier that 
looks like a scaled variant of the integral version, namely Stein's operator in \eqref{e:stein}.  But, the  multiplier  has several 
distinguished frequency points, as described in \S\ref{s:max}.  This is an additional arithmetic 
feature identified by Bourgain. 
Establishing the oscillatory variant of Bourgain's estimate is  one of the innovations of this paper.  
However, as the rational approximate to $ \lambda $ changes, so do the multipliers, forcing another 
application of  the powerful inequality of Theorem~\ref{KEY}.  
This can however only be done in a controlled number of times. 
The \emph{arithmetic} part of definition of Minkowski dimension is then decisive.  

We believe  Pierce's conjecture is   true, but have very limited techniques 
at our disposal to prove it.  On the other hand, there are many potential variants of the approach 
established in this paper, and it would be interesting to explore some of these.  
In a different direction,  the pivotal function $ M (\lambda , \beta )$ is the Green's function for a periodic Schr\"odinger equation. 
It has several beautiful properties, including functional relations,   \cite {MR1650976,MR3241850}.  We would not be surprised 
to learn that these properties are relevant to Pierce's conjecture.

\subsection{Acknowledgments}
This project began at the AIM workshop ``Carleson theorems and multilinear operators" held in May 2015. The authors are grateful to all the organizers of the conference, and to Lillian Pierce in particular, who brought the conjecture above to their attention. Additional thanks are due to Michael Christ, Kevin Henriot, Izabella Laba, Xiaochun Li, Victor Lie, Camil Muscalu, and Jill Pipher for many helpful discussions on the matter, and to Terence Tao for his continued support.

\section{Preliminaries}


\subsection{Notation}
As previously mentioned, we let $e(t):= e^{2\pi i t}$. Throughout, $c$ will denote a small number whose precise value may differ from line to line, and $0 < \epsilon \ll 1$ will denote a sufficiently small positive quantity.

We will also fix throughout a smooth dyadic resolution of the function $\frac{1}{t}$. Thus, 
\begin{equation}\label{e:1/t}
 1/t = \sum_{k \in \Z} \psi_k(t) := \sum_{k \in \Z} 2^{-k} \psi(2^{-k} t), \qquad t\neq 0, 
\end{equation}
where $ \psi $ is a smooth  odd function satisfying $ \lvert  \psi (x) \rvert  \leq \mathbf 1_{ [-1/4 , 1]} (\lvert  x\rvert )$. 
We will be concerned with the regime $|t| \geq 1$, we will restrict our attention to
\[ \sum_{k \geq 0} \psi_k(t),\]
which agrees with $\frac{1}{t}$ for $|t| \geq 1$. 

Finally, since we will be concerned with establishing a priori $\ell^2(\Z)$- or $L^2(\RR)$-estimates in this paper, we will restrict every function considered to be a member of a ``nice'' dense subclass: each function on the integers will be assumed to have \emph{finite support}, and each function on the line will be assumed to be a \emph{compactly supported, indefinitely differentiable} function.

We will make use of the modified Vinogradov notation. We use $X \lesssim Y$, or $Y \gtrsim X$ to denote the estimate $X \leq CY$ for an absolute constant $C$. We use $X \approx Y$ as
shorthand for $Y \lesssim X \lesssim Y $.
We also make use of big-O notation: we let $O(Y)$ denote a quantity that is $\lesssim Y$.

\subsection{Transference}
The Fourier transform on $  f \in L^2(\mathbb R )$ is  defined by 
\begin{equation*}
\mathcal F _{\mathbb R } f (\xi ) =  \int_{\RR} f (x) e (- \xi x)\; dx . 
\end{equation*}
This is a unitary map on $ L ^2 (\mathbb R )$, with inverse
\begin{equation*}
\mathcal F _{\mathbb R }^{-1} g (x ) =  \int_{\RR} g (\xi) e ( \xi x)\; dx . 
\end{equation*}

{Recall that the fundamental domain for the torus is  $[-1,1] \subset \RR$.} 
The Fourier transform on $ f \in \ell ^2 (\mathbb Z )$ is defined by 
\begin{equation*}
\mathcal F _{\mathbb Z } f (\beta ) = \sum_{n\in \Z} f (n) e (- \beta n).  
\end{equation*}
This is a unitary map from $ \ell ^2 (\mathbb Z )$ to $ L ^{2} ( \mathbb T )$, with inverse Fourier transform
\begin{equation*}
\mathcal F^{-1} _{\mathbb Z } g (n ) = \int_{\TT} g(\beta) e ( \beta n) \ d\beta.  
\end{equation*} 

The main analytic tool of our paper (\S 3 below) is a maximal inequality over convolutions, which is seen as a maximal multiplier theorem:
\begin{equation} \label{e:ZZ}
\Bigl\lVert \sup _{\lambda \in \Lambda } \bigl\lvert  \mathcal F _{\mathbb Z } ^{-1} (L (\lambda , \beta ) \mathcal F _{\mathbb Z } f) \bigr\rvert \Bigr\rVert _{\ell ^2 (\mathbb Z )} \lesssim \lVert f\rVert _{\ell ^2 (\mathbb Z )}.  
\end{equation}
Here, $\{ L(\lambda,\beta) : \lambda \in \Lambda \}$ are compactly supported inside $\TT$.

As such, we can view these multipliers as living in a common fundamental domain for the torus, $A \subset [-1,1] \subset \RR$, $A \cong \RR/\Z$, inside the real line. The following transference lemma due to Bourgain endorses this change of perspective.

\begin{lemma} [Lemma 4.4 of \cite{MR1019960}]
\label{l:transference}
Suppose $\{ m(\lambda, - ) : \lambda \in \Lambda \}$ are (a countable family of) uniformly bounded multipliers on $A \subset [-1,1]$, a fundamental domain for $\TT \cong \RR/ \Z $.
If there is a constant $ C $ such that 
\[ 
\Bigl\lVert \sup _{\lambda \in \Lambda } \bigl\lvert  \mathcal F _{\mathbb R } ^{-1} (m(\lambda , \xi ) \mathcal F _{\mathbb R } f) \bigr\rvert \Bigr\rVert _{L ^2 (\mathbb R )} \leq C  \lVert f\rVert _{L ^2 (\mathbb R )},
\]
then, there holds on $ \ell ^2 (\mathbb Z )$, 
\[ 
\Bigl\lVert \sup _{\lambda \in \Lambda } \bigl\lvert  \mathcal F _{\mathbb Z } ^{-1} (m (\lambda , \beta ) \mathcal F _{\mathbb Z } f) \bigr\rvert \Bigr\rVert _{\ell ^2 (\mathbb Z )} \lesssim  C \lVert f\rVert _{\ell ^2 (\mathbb Z )}. 
\]
\end{lemma}

We will henceforth bound our maximal operators on $ \mathbb R $, finding that the dilation and translation structure 
useful.   And, so we write $ \mathcal F = \mathcal F _{\mathbb R }$ below, 
and sometimes write $ \widehat f = \mathcal F f  $.

\subsection{A Technical Estimate}
This Lemma exhibits one way in which the small Minkowski dimension is used.  
It concerns a maximal multiplier estimate, in which one has a trade off between a relatively small 
$ L ^{\infty }$ estimate on the the multipliers, and a relatively large derivative estimate.

\begin{lemma}\label{SOB}
Suppose that $\{ m(\lambda,-) : \lambda \in [0,1]\}$ is a family of multipliers, that is differentiable in $\lambda$ with
\begin{equation}\label{e:SOB}
\lVert  m(\lambda, \beta ) \rVert_ \infty  \leq a, 
\qquad \|\partial_\lambda m(\lambda, \beta )\|_ \infty   \leq A.
\end{equation}
If  $ \Lambda $ is of  Minkowski dimension $d >0 $, then
\begin{equation}\label{e:dimLam}
\Bigl\lVert 
\sup _{\lambda \in \Lambda } \bigl\lvert  \mathcal F ^{-1}  ( m (\lambda , \beta ) \widehat f (\beta ) )\bigr\rvert
\Bigr\rVert_2 \lesssim C _{\Lambda } ^{1/2} 
(a + a^{ 1 - \frac d2}A^{\frac d2}  ) \lVert f\rVert_2 
\end{equation}
Above $ C _{\Lambda } = \sup _{0 < \delta  <1} \delta ^{-d} N (\delta )$, where 
$ N (\delta )$ is the smallest number of intervals of length $ \delta $ needed to cover $ \Lambda $. 
\end{lemma}

We will apply this Lemma in a setting where $   2 ^{- \epsilon j} \simeq a < A < a ^{- \tau }$, where $ 0< \epsilon < 1$  
and $ \tau \approx 2/ \epsilon >1$ are  fixed, while   $ j\in \mathbb N $ is arbitrary. 
Note that if we take $ d= 1/ \tau \approx \epsilon /2$, we would then have 
\begin{equation}\label{e:tau}
\Bigl\lVert 
\sup _{\lambda \in \Lambda } \bigl\lvert  \mathcal F ^{-1}  ( m (\lambda , \beta ) \widehat f (\beta ) )\bigr\rvert
\Bigr\rVert_2 \lesssim C _{\Lambda }  a ^{\frac14 - \frac{\epsilon}{4}} \lVert f\rVert_2 
\end{equation}
This will be a summable estimate in  $ j\in \mathbb N $, since $ a \simeq 2 ^{- \epsilon j}$.

The Lemma only requires the assumption of upper Minkowski dimension $ d$ on $ \Lambda $,  the 
arithmetic condition being used later.  
Variants of this Lemma  have been observed before, but as we do not know a easy reference, we include a proof.

\begin{proof}  
We will not need the `arithmetic' part of Definition~\ref{d:amk} in this proof, 
and for the purposes of this argument, write 
\begin{equation*}
F _{\lambda } f :=  \mathcal F ^{-1}  ( m (\lambda , \beta ) \widehat f (\beta ) )
\end{equation*}
For a choice of $ 0 < \delta < 1$ to be chosen, 
let $ \Lambda _ \delta  $ be the centers of a covering of $ \Lambda $ by intervals of length $ \delta $, 
so that the cardinality of $ \Lambda _{\delta }$ is at most $ C _{\Lambda } \delta ^{-d}$. 
For a fixed $ \lambda _0 \in \Lambda _{\delta }$, let $ \lambda _{\pm} = \lambda _{\delta } \pm \delta/2 $. 
Observe that for  $ \lambda _{-} \leq \lambda \leq \lambda _{+}$, there holds 
\begin{align*}
\lvert  F _{\lambda } f \rvert& \leq  \lvert  F _{\lambda_0 } f  \rvert   
+
 \int _{\lambda_- } ^{ \lambda _+}    \left| \partial _{ \mu  } F _{\mu  } f  \right| \; d \mu 
\end{align*}

We compute $ L ^2 $ norm of the supremum over $ \lambda \in \Lambda $.  For the first term on the right above, 
dominating a supremum by $ \ell ^2 $-sum we have 
\begin{align}  \label{e:D1}
\sum_{\lambda _0 \in \Lambda _{\delta }} \lVert  F _{\lambda_0 } f  \rVert_2 ^2 
\leq C _{\Lambda }\delta ^{-d} a  ^2  \lVert f\rVert_2 ^2 . 
\end{align}
Here, one should note that we have used the Minkowski dimension to bound the cardinality of $ \Lambda  _{\delta }$ 
by $ C _{\Lambda} \delta ^{-d}$.    For the second term, we use the estimate \eqref{e:SOB} and Cauchy-Schwarz 
to bound 
\begin{align}
\sum_{\lambda _0 \in \Lambda _{\delta }}  
\Bigl\lVert  \int _{\lambda _- } ^{\lambda _+}  \lvert   \partial _{ \mu  } F _{\mu  } f   \rvert \; d \mu \Bigr\rVert_2 ^2 
& \leq \delta \sum_{\lambda _0 \in \Lambda _{\delta }}
 \int _{\lambda _- } ^{\lambda _+}  \lVert   \partial _{ \mu  } F _{\mu  } f   \rVert_2 ^2  \; d \mu 
 \\ \label{e:D2}
 & \leq  C _{\Lambda } \delta ^{2-d} A ^2  \lVert  f\rVert_2 ^2 .    
\end{align}
One should note that we have gained a power of $ \delta ^2 $ above, before applying the Minkowski dimension assumption. 

The Lemma is proved by choosing $ \delta >0$ so that the bounds in \eqref{e:D1} and \eqref{e:D2} are equal. 
In the less interesting case where $ 0 < A \leq a$, we take $ \delta =1$, to complete the case. 
Otherwise, we take $ \delta = a/A$ to complete the proof.  
\end{proof}

\subsection{A Second Technical Estimate}

If $ \sigma $ is a Schwartz function on $ \mathbb R $ which vanishes at the origin and at infinity, then as is well known, 
\begin{equation*}
\bigl\lVert \sup _{\lambda >0}  \lvert  \mathcal F ^{-1} ( \sigma (\xi  \lambda ) \widehat f (\sigma ) )\rvert\, \bigr\rVert_2 
\lesssim \lVert f\rVert_2. 
\end{equation*}
We need a certain quantification of this fact.

We suppose that $ \{\sigma (\lambda,  \cdot  ) \;:\; 0< \lambda < \infty \}$ are  functions on $ \mathbb R $ that satisfy 
this estimate for $  \lambda \in (0, \infty )$.   Below  $ 0 < c < C < \infty $ 
are fixed constants.  For a fixed $ \rho  \geq 1$
\begin{equation}\label{e:sig1}
\lvert\sigma (\lambda, \xi )\rvert  
\lesssim 
\begin{cases}
\frac { \lvert  \xi \rvert } {\rho ^2 \sqrt \lambda  }      &  \lvert  \xi \rvert    < c   \rho  \sqrt \lambda  
\\
1/ \rho    &  c  \rho  \sqrt \lambda   < \lvert  \xi \rvert   < C \rho  \sqrt \lambda  
\\
\frac {\sqrt \lambda } {   \lvert  \xi  \rvert } &  C   \rho  \sqrt \lambda     <\lvert  \xi \rvert  
\end{cases}.  
\end{equation}
Set $ F _{\lambda } f (x) =  \mathcal F ^{-1} ( \sigma ( \lambda, \xi  ) \widehat f (\xi  )) (x) $. 

\begin{lemma}\label{l:sig}  Let $ \rho \geq 1$.  Assume that $   \{\sigma (\lambda,  \cdot  ) \;:\; 0< \lambda < \infty \}$ satisfy \eqref{e:sig1}. 
Assume also that $ \sigma (\lambda , \cdot )  $ is $ C ^{1}$ on each interval $ \{ (2 ^{k}, 2 ^{k+1} )\;:\; k\in \mathbb Z \}$, 
with moreover    $  \{ (\lambda /\rho^2 ) \partial _{\lambda }  \sigma (\lambda,  \cdot  )\}$ satisfying the inequalities \eqref{e:sig1}.  
Then, there holds 
\begin{equation}\label{e:sig2}
\bigl\lVert \sup _{\lambda >0}   \lvert  F _{\lambda } f\rvert \,  \bigr\rVert_2 
\lesssim \lVert f\rVert_2. 
\end{equation}
\end{lemma}

Above, we do not impose any condition on the derivative of $ \sigma (\lambda , \cdot )$ at the points of discontinuity 
$ \lambda = 2 ^{k}$ for $ k\in \mathbb Z $. 

\begin{proof}
Let $ \Lambda \subset (0, \infty )$ be such that each integer $ k$,  we have $ 2 ^{k} \in \Lambda $, 
the set $  \Lambda \cap [2 ^{k}, 2^{k+1})$ is a maximal $ 2 ^{k}/ \rho ^2$-net of  cardinality at most $ 2 \rho  ^2 $.  
For $  \lambda \in (0, \infty )$, let $ \lambda_\mp \in \Lambda $ be the largest/smallest element of $ \Lambda $ that 
is smaller/larger than $   \lambda $. Then, $ 0  <  \lambda_+  -  \lambda_- < 2 \lambda / \rho ^2  $.  
 For any $ \lambda $, we will have $ \lambda _-, \lambda , \lambda _+ \in [2 ^{k}, 2 ^{k+1}]$ for some integer $ k$. 
 This lets us estimate 
\begin{align}    
\lvert  F _{  \lambda } f \rvert 
&\leq \lvert   F _{\lambda _-} f \rvert + \int _{\lambda _-} ^{\lambda _+} \lvert  \partial _{\lambda } F _{\lambda } f\rvert\; d \lambda \\
&\leq \left(\sum_{\lambda_{-} \in \Lambda} |F_{\lambda_-} f|^2 \right)^{1/2}  +
\left(\sum_{\lambda_{-} \in \Lambda} \left(\int_{\lambda_-}^{\lambda_+} |\partial_\lambda F_{\lambda} f| \ d\lambda \right)^2 \right)^{1/2} \\  \label{e:sig3}
&\lesssim \left(\sum_{\lambda_{-} \in \Lambda} |F_{\lambda_-} f|^2 \right)^{1/2} + 
\left(\sum_{\lambda_{-} \in \Lambda} \frac{\lambda_-}{\rho^2} \cdot \int_{\lambda_-}^{\lambda_+} |\partial_\lambda F_{\lambda} f|^2 \ d\lambda \right)^{1/2} 
\end{align}
by an application of Cauchy-Schwarz. Note that this final expression is independent of $\lambda$, so we obtain the conclusion upon estimating the right-hand side of the foregoing in $L^2$. 

For the first term in \eqref{e:sig3}, we have 
\begin{align*}
\sum_{\lambda \in \Lambda} \lVert F_{\lambda} f\rVert_2 ^2 
& 
= \int \lvert  \widehat f (\xi ) \rvert ^2 
\sum_{\lambda \in \Lambda} \lvert  \sigma (\lambda , \xi )\rvert ^2 \; d \xi .  
\end{align*}
We estimate the $ L ^{\infty  } (d \xi )$ norm  of the sum above, by showing that for  $ k \in \mathbb Z $, 
\begin{equation}  \label{e:sig4}
 \sum_{\substack{\lambda \in \Lambda\\ 2 ^{k} \leq  \frac{ \lvert  \xi \rvert }{ \rho \sqrt {\lambda }}   < 2 ^{k+1}    }} \lvert  \sigma (\lambda , \xi )\rvert ^2
\lesssim   \min \{  2 ^{2k}, 2 ^{-2k}\}
\end{equation}
The bound above is summable  in $ k\in \mathbb Z $ to a constant, completing the bound for the first term in \eqref{e:sig3}.

We divide the proof of \eqref{e:sig4} into three cases, $ k< -c'$, $ -c'\leq k \leq c'$, and $ k >c'$, where 
$ c'>0$ is a fixed constant, depending upon the constants $ 0 < c < C < \infty $ in \eqref{e:sig1}.  
Using \eqref{e:sig1}, and $ \lvert  \Lambda \cap  [ 2 ^{j}, 2 ^{j+1})\rvert \leq \rho ^2  $ for all integers $ j$,  
we have 
\begin{equation*}
\textup{LHS \eqref{e:sig4}} \lesssim  
 \rho ^2  \sup _{ 2 ^{k} \leq  \frac{ \lvert  \xi \rvert }{ \rho \sqrt {\lambda }}   < 2 ^{k+1}} 
 \begin{cases}
\Bigl( \frac { \lvert  \xi \rvert  } {\rho ^2 \sqrt \lambda  }    \Bigr) ^2  &  k < -c'
\\
\rho ^{-2}   &  -c \leq k \leq c 
\\
 \Bigl(\frac {\sqrt \lambda } {   \lvert  \xi  \rvert }  \Bigr) ^2 & c < k
 \end{cases}.
\end{equation*}
Then, \eqref{e:sig4} follows by inspection. 

For the second term in \eqref{e:sig3}, we have 
\begin{align*}
\sum_{\lambda_{-} \in \Lambda} \frac{\lambda_-}{\rho^2} \cdot \int_{\lambda_-}^{\lambda_+} |\partial_\lambda F_{\lambda} f|^2 \ d\lambda  
=& 
\int \lvert  \widehat f (\xi )\rvert ^2 
\sum_{\lambda_- \in \Lambda} \Bigl(\frac{\lambda_-}{\rho^2} \Bigr) ^2 
\sup _{ \lambda_- \leq \lambda \leq \lambda _+} \lvert  \partial_ \lambda  \sigma (\lambda , \xi )\rvert ^2 \; d \xi .  
\end{align*}
But, then under our assumptions on $  \partial_ \lambda  \sigma (\lambda , \xi )$, the analysis of the first term in \eqref{e:sig3} applies with no changes to the argument. 
\end{proof}

\section{A Key Maximal Inequality} \label{s:max}

We present and prove a key maximal inequality. It is an extension of this inequality of Bourgain, 
the harmonic analytic core of the proof of the arithmetic ergodic theorems.  
Define $\Phi _\lambda f := \varphi_\lambda * f$, where $\varphi$  Schwartz function satisfying
\begin{equation}\label{e:1phi}
 \mathbf 1_{[-1/8,1/8]} \leq \hat{\varphi} \leq \mathbf 1_{[-1/4,1/4]}, 
\end{equation}  
and $ \varphi _{\lambda } (y) = \varphi (y/\lambda ) / \lambda $.  
The inequality $ \lVert \sup _{\lambda >0}  \lvert  \Phi _{\lambda }f\rvert \rVert_2 \lesssim \lVert f\rVert_2$ 
is a variant of the Hardy-Littlewood maximal inequality in which zero is a distinguished frequency mode. 

Bourgain's version has several distinguished frequency modes.  
Let $ \{ \theta _ n \;:\; 1\le n \le N\}$ be points in $ \mathbb R $ which are  $ \tau $-separated, in that 
$ \lvert  \theta _m - \theta _n\rvert> \tau >0 $ for $ m\neq n$.  Define a maximal operator by 
\begin{equation}\label{e:M}
Mf(x):= \sup_{ \lambda > 1/\tau } \left| \sum_{n=1}^N \textup{Mod}_{\theta_n} \Phi _\lambda(  \textup{Mod}_{-\theta_n} f)  \right|, 
\end{equation}
where $\textup{Mod}_\theta f (x) := e(\theta x) f(x)$. 
Trivially, the norm of $ M$ is dominated by $ N$.  The key observation is that that 
this trivial bound can be improved to the much smaller term $ {\log}^2 N$.

\begin{theorem}\cite[Lemma 4.13]{MR1019960}\label{t:BMax}
For all $ N \geq 2$,  $ 0 < \tau < \infty $ and   $ \tau $-separated points  $ \{ \theta _ n \;:\; 1\le n \le N\}$, there holds 
\[ \|Mf \|_{L^2(\RR)} \lesssim {\log}^2 N \|f\|_{L^2(\RR)}.\]
\end{theorem}

We prove an extension of this inequality, with the averages replaced by 
oscillatory singular integrals.  
Define 
\begin{equation} \label{e:Tlambda}
T_\lambda f(x) :=   \sum_{k \geq k_0} \int e(\lambda t^2) \psi_k(t) f(x-t) \; dt. 
\end{equation}
Above, $2 ^{ k_0} \geq 1/ \tau $, and the dependence on $ k_0$ is suppressed.  
Stein's inequality  \cite{MR1364908} implies  that $ f\mapsto  \sup_{0 < \lambda \leq  \tau^2 } |T_\lambda f|$ is a 
 bounded operator on $L^2(\RR)$.    For $ \{\theta _n \;:\; 1\leq n \leq N\}$ that are  $ \tau $-separated as in Theorem~\ref{t:BMax}, 
 define 
\begin{equation}\label{e:T}
 Tf(x) := \sup_{0< \lambda \leq  \tau ^{2} } \left| \sum_{n=1}^N \textup{Mod}_{\theta_n} T_\lambda \left( \varphi _{\tau } * \textup{Mod}_{-\theta_n} f \right) \right|.
\end{equation}
This definition matches that of \eqref{e:M}, except that there is an additional convolution with 
$  \varphi _{\tau } $ as in \eqref{e:1phi}.  
Note that the supremum is over all $ 0< \lambda < \tau ^2  $.  
The main result of this section is as follows.  

\begin{theorem}\label{KEY}
For all  $ 0< \tau < 2 ^{-k_0} < \infty $,     $ N \geq 2$, and  $ \tau $-separated points  $ \{ \theta _ n \;:\; 1\le n \le N\}$, 
this inequality holds.  
\[ \| T f\|_{L^2(\RR)} \lesssim {\log}^2 N \| f\|_{L^2(\RR)}.\]
\end{theorem}

We will apply this with $ \tau \approx 2 ^{-s}$, and $ \tau $-separated 
points in $ [0,1]$,  so that $ N \leq 2 / \tau $.    
An additional gain of $ 2 ^{- \epsilon s}$ will come from a well known estimate  on complete Gauss sums.

\subsection{Proof Overview}
A potential path to a proof is through Bourgain's innovative proof of Theorem~\ref{t:BMax}. 
The latter depends upon subtle properties of averages, such as the quadratic variational result. 
Such estimates for oscillatory integral don't seem to hold, however. 
Thus, our path is to invoke Bourgain's inequality, and control residual terms by square function arguments.

Using the dilation structure of $\RR$, it suffices to consider the case of $ \tau =1$ in Theorem~\ref{KEY}.   
 (We omit the straight forward proof of this reduction.)  
The operator $ T _{\lambda }$ in \eqref{e:Tlambda} is decomposed as follows. In the integral 
\begin{equation*}
\int e(\lambda t^2) \psi_k(t) f(x-t) \; dt, 
\end{equation*}
the variable $ t$ is approximately $ 2 ^{k}$ in magnitude.  And, we will decompose the operator and maximal function so that 
$ \lambda t ^2 \approx \lambda 2^{2k}$ is approximately constant.  

To this end,write 
\begin{equation}\label{e:Tkl}
\begin{split}
 T_\lambda f(x)=&\sum_{l \in \Z}  \sum_{ k \geq k_0}  \int e(\lambda t^2) \psi_k(t) f(x-t) \; dt \cdot    \mathbf 1_{[1,2)}(\lambda 2 ^{2k-l}) \\
& \qquad =: \sum_{l \in \Z} \int e(\lambda t^2) \psi_{k,\lambda,l}(t) f(x-t) \; dt  =: \sum_{l \in \Z} T_\lambda^l f(x). 
\end{split}
\end{equation}
Here $ \psi_{k, \lambda,l} (t) =  \psi_k (t)  $ 
and $ k = k (\lambda ,l)$ is the unique choice of $ k$ so that $1 \leq \lambda 2 ^{2k-l} <2$.

The supremum over $ \lambda $ is then divided into four separate cases, according to the relative size of 
$ \ell $ and $ N$.  
\begin{align}
Tf &\leq  \sup_{0\leq \lambda \leq 1} \left| \sum_{n=1}^N \textup{Mod}_{\theta_n}  \sum_{l < -10 \log_2 N} T_\lambda^l  ( \varphi * \textup{Mod}_{-\theta_n}f)\right| 
\label{e:1}\\
& \qquad + \sum_{l = -10 \log_2 N}^ 0 \sup_{0\leq \lambda \leq 1} 
\left| \sum_{n=1}^N \textup{Mod}_{\theta_n} T_\lambda^l ( \varphi * \textup{Mod}_{-\theta_n}f)\right|  
\label{e:2}\\
& \qquad \qquad + \sum_{l = 1}^ {10 \log_2 N}  \sup_{0\leq \lambda \leq 1} 
\left| \sum_{n=1}^N \textup{Mod}_{\theta_n} T_\lambda^l ( \varphi * \textup{Mod}_{-\theta_n}f)\right| 
\label{e:3}\\
& \qquad \qquad \qquad+ \sum_{ l > 10 \log_2 N} \sum_{n=1}^N \sup_{0\leq \lambda \leq 1} \left| T_\lambda^l (\varphi * \textup{Mod}_{-\theta_n}f) \right|. 
\label{e:4}
\end{align}
The first term is a relatively easy instance of Bourgain's inequality Theorem~\ref{t:BMax}. 
The next two terms are also a consequence of Bourgain's inequality, with an additional argument. 
The last term is of a simpler nature.  

\subsection{Case One: $l < -10 \log_2 N$} {}

The term to bound is \eqref{e:1}. The essence of this case is that the parameter $ \lambda $ is 
so small that the oscillatory term is essentially constant, and thus, the role of $ \lambda $ is only  
to introduce a truncation on a Hilbert transform.    Let us formulate this last point as a 
Lemma.  Define 
\begin{equation}\label{e:Hk}
H^ j f (x) :=\sum_{  k = j} ^{\infty }  \int  \psi_k(t) f(x-t) \; dt. 
\end{equation}
We will use subscripts to `single scale' operations, and superscripts to denote sums over the same.

\begin{lemma}\label{l:H} 
For all $ N \geq 2$, and 1-separated points  $ \{ \theta _ n \;:\; 1\le n \le N\}$, there holds 
\begin{equation}\label{e:H}
\biggl\lVert  \sup _{j \geq 0}  \biggl\lvert  \sum_{n=1} ^{N} 
\textup{Mod} _{\theta _n}  H^j \varphi \ast (\textup{Mod} _{- \theta _n} f )
\biggr\rvert \biggr\rVert_2 \lesssim \log ^2 N \cdot \lVert f\rVert_2. 
\end{equation}
\end{lemma}

\begin{proof}
This is a consequence of Bourgain's inequality.  Define 
\begin{equation*}
g =   \sum_{n=1} ^{N} 
\textup{Mod} _{\theta _n}  H( \varphi \ast \textup{Mod} _{- \theta _n} f )
\end{equation*}
where $ H$ is the usual (un-truncated) Hilbert transform. Then, $ \lVert g\rVert_2\leq \lVert f\rVert_2$, 
and there holds 
\begin{equation*}
\lVert M g \rVert_2 \lesssim \log ^2 N \lVert f\rVert_2
\end{equation*}
where $ M $ is maximal operator in Theorem~\ref{t:BMax}.  
Now, recall that $ \Phi _ \lambda f = \varphi _{\lambda } \ast f$ is an average on scale $ \lambda $, with $ \varphi $ 
as in \eqref{e:1phi}.  Then, we can compare the maximal function $ Mg$ to the one in \eqref{e:H} by a square 
function in $ j\geq 0$.   

The point is that averages of Hilbert transforms and truncations of Hilbert transforms are 
essentially the same object.   
In the case of $ N=1$,  for any function $ \phi \in L ^2 (\mathbb R )$, we have 
by Plancherel, 
\begin{align*}
\lVert    H^j \phi - \Phi _{2 ^{j}} H \phi  \rVert_2 ^2 
&= 
\int \lvert \widehat \phi  (\xi )\rvert ^2 
 \bigl\lvert  \Psi (\xi 2 ^{-j}) -  \widehat \varphi (\xi 2 ^{-j}) \textup{sgn} (\xi ) \bigr\rvert ^2 \; d \xi 
\end{align*}
Here, $ \Psi $ is given by 
\begin{equation*}
  \Psi (\xi ) = \sum_{j\geq 1} \widehat \psi (2 ^{j}\xi ),  
\end{equation*}
which satisfies the estimates  below, uniformly in $ \xi \in \mathbb R $. 
\begin{equation*}
\lvert  \Psi (\xi )- \widehat \varphi (\xi ) \textup{sgn} (\xi ) \rvert 
\lesssim \min \{ \lvert  \xi \rvert, \lvert  \xi \rvert ^{-1}   \}. 
\end{equation*}
It follows that 
\begin{equation*}
\sum_{j\in \mathbb Z } \lVert    H^j \phi - \Phi _{2 ^{j}} H \phi  \rVert_2 ^2  \lesssim \lVert \phi \rVert_2 ^2 .  
\end{equation*}

Therefore, when we have the 1-separated frequencies $ \{\theta _n\}$, and convolution with respect to $ \varphi $, 
to localize the frequency around each $ \theta _n$,  it follows that 
\begin{align*}
\sum_{j\geq 0} 
\biggl\lVert   \sum_{n=1} ^{N} 
\textup{Mod} _{\theta _n} \{  H^j( \varphi \ast \textup{Mod} _{- \theta _n} f ) -  \Phi _{ 2 ^{j}} \textup{Mod} _{- \theta _n} g\} 
\biggr\rVert_2 ^2 \lesssim \lVert f\rVert_2 ^2 . 
\end{align*}
This proves inequality \eqref{e:H}.

\end{proof}

We are to bound the term in \eqref{e:1},  and to conform to notation set therein, let 
\begin{equation*}
\tilde H^ j f (x) :=\sum_{  k=k_0} ^{  j }  \int  \psi_k(t) f(x-t) \; dt. 
\end{equation*}
which is the complementary sum to the one in \eqref{e:Hk}.  As an immediate corollary to \eqref{e:H}, 
there holds 
\begin{equation}\label{e:tildeH}
\biggl\lVert  \sup _{j \geq 0}  \biggl\lvert  \sum_{n=1} ^{N} 
\textup{Mod} _{\theta _n}  \tilde H^j \varphi \ast (\textup{Mod} _{- \theta _n} f )
\biggr\rvert \biggr\rVert_2 \lesssim \log ^2 N \cdot \lVert f\rVert_2. 
\end{equation}
Extend the notation above to 
\[ \aligned
R_\lambda f(x) &:= \sum_{ k \geq k_0  \;:\;  2^{2k} \leq \frac{2}{N^{10} \lambda} } \int  \psi_k(t) f(x-t) \; dt. \endaligned \]
Then, 
\begin{align}
\eqref{e:1} \leq& 
\sup_{0\leq\lambda \leq 1}  \left| \sum_{n=1}^N \textup{Mod}_{\theta_n} R_\lambda ( \varphi * \textup{Mod}_{-\theta_n}f)\right|
\\  \label{e:g1}
&\quad + \sum_{n=1} ^{N}  
\sup_{0\leq \lambda \leq 1}   
\biggl\lvert  \biggl\{ R _{\lambda } - 
 \sum_{l < -10 \log_2 N} T_\lambda^l \biggr\}  ( \varphi * \textup{Mod}_{-\theta_n}f)
\biggr\rvert . 
\end{align}
The first term on the right  is an instance of \eqref{e:tildeH}.  Thus, it remains to bound the second term above.

But,  the convolutions kernel for the difference of the two operators  in \eqref{e:g1} 
is easily controlled, uniformly in $ \theta _n$.    
It is 
\begin{align*}
\sum_{ k \geq k_0 : 2^{2k} \leq \frac{2}{N^{10} \lambda} }  
\lvert   e ( \lambda t ^2  ) -1 \rvert  \cdot  \lvert  \psi _k (t)\rvert 
& \lesssim   
\sum_{ k \geq k_0 : 2^{2k} \leq \frac{2}{N^{10} \lambda} }       \lambda t ^2   \cdot \lvert  \psi _k (t)\rvert  
\\
& \lesssim \lambda 2 ^{2k_1} 
\sum_{ k \geq k_0 : 2^{2k} \leq \frac{2}{N^{10} \lambda} }       (t/2 ^{k_1})^2   \cdot \lvert  \psi _k (t)\rvert  
\\
& \lesssim N ^{-10}  2 ^{-k_1}  \mathbf 1_{ [ -2 ^{k_1 +2}, 2 ^{k_1+2} ]} (y)
\end{align*}
where $ k_1$ is the largest integer $ k_1 \geq k_0$ such that $ 2^{2k} \leq \frac{2}{N^{10} \lambda}$.  
That is, the difference  of operators,  for a fixed $ \theta _n$  is bounded by 
\begin{align*}
\sup_{0\leq \lambda \leq 1}   
\biggl\lvert  \biggl\{ R _{\lambda } - 
 \sum_{l < -10 \log_2 N} T_\lambda^l \biggr\}  ( \varphi * \textup{Mod}_{-\theta_n}f)
\biggr\rvert \lesssim N ^{-10} M _{\textup{HL}} \lvert  f\rvert 
\end{align*}
 where $M_{\textup{HL}}$ denotes the Hardy-Littlewood maximal operator.
 This estimate is trivially summed over $ 1\leq n \leq N$, completing the bound for \eqref{e:g1}.

\subsection{Cases Two and Three: $|l| \leq 10 \log_2 N$} {}

The terms to control are \eqref{e:2} and \eqref{e:3}, in which the sum over $ l$ is limited to  $|l| \leq 10 \log_2 N$. 
We prove the bound below, in which $ l$ is held constant.  

\begin{lemma}
For any $|l| \leq 10 \log_2 N$,  
\begin{equation}\label{e:case23}
\begin{split}
\biggl\| \sup_{0\leq\lambda \leq 1}  
\biggl| \sum_{n=1}^N& \textup{Mod}_{\theta_n} T_\lambda^l ( \varphi * \textup{Mod}_{-\theta_n}f)\biggr| \biggr\|_{L^2(\RR)}
\\ &\lesssim \left( 1 + \min\{2^l, 2^{-l} \} \cdot {\log}^2 N \right) \|f\|_{L^2(\RR)}. 
\end{split}
\end{equation}
\end{lemma}

Notice that the sum over $ \lvert  l\rvert \leq 10 \log_2 N $ of the constants on the right lead to a bound of the form 
$ C\log ^2 N$.  

Bourgain's inequality, Theorem~\ref{t:BMax}, remains crucial, followed by an application of Lemma~\ref{l:sig}.
Much of the argument is common to both cases of $ l \leq 0$ and $ l > 0$. 
First, some notation: Recalling \eqref{e:Tkl}, with $ 1\leq \lambda 2 ^{2k- l } <2$, define 
\[ \mu(k,\lambda, l) :=  \int e(\lambda t^2) \psi_{ k  }(t)  \; dt .\]
This is the zero Fourier mode of the multipliers in question.  
Observe that 
\begin{equation}\label{e:mu<}
 |\mu(k,\lambda,l) | \lesssim \min\{ 2^{l}, 2^{- l} \}.
\end{equation}
\begin{proof} 
By the definition of $ \mathbf 1_{[1,2)} $, we restrict attention to $ \lambda \approx 2 ^{-2k+l}$. 
For $ l >0$,  on the support of $ \psi _k$, the derivative of $ \lambda t ^2 $ is at least $ c \lambda \lvert  t\rvert \gtrsim 2 ^{l} $. 
Thus, a single integration by parts proves the estimate. 

For $ l \leq 0$, since $ \psi $ is odd, clearly $ \mu (k,0, l) =0$.  Now, the claim follows from the bound 
\begin{equation*}
\Bigl\lvert \partial_ \lambda  \int e (\lambda t ^2 ) \psi _k (t)\; dt  \Bigr\rvert \lesssim 2 ^{2k},   
\end{equation*}
subject to the side condition $ 0< \lambda \lesssim  2 ^{-2k+l}$.  
Now, the partial in $ \lambda $ brings down an $ i t ^2 $ from the exponential, which immediately gives the bound above. 
\end{proof}

The purpose of this next definition  is to permit us to apply Bourgain's inequality. 
With $\varphi$ a Schwartz function as in \eqref{e:1phi}, 
define, 
\begin{equation}\label{e:phikl}
\phi_{k,\lambda,l}(t) := e(\lambda t^2) \psi_{k, \lambda , l} (t)    - \mu(k,\lambda,l) \varphi_{k_l}(t).  
\end{equation}
It is essential to note that we are subtracting off the zero frequency mode above. 
Also,   there is an adjustment in scales that is made for technical reasons:  
\begin{equation}\label{e:kl}   
k _{l} := 
\begin{cases}
k  & l \leq 0 \\ k-l & l >0 
\end{cases}. 
\end{equation}

Expand $  T_\lambda^l $ into 
\[ T_\lambda^l f(x) = \Phi_{k,\lambda,l} f(x) + \mu(k,\lambda,l) \varphi_{k_l}*f(x),\]
where $ \Phi_{k,\lambda,l} f(x) := \phi_{\lambda,l}*f(x)$.  
The left side in \eqref{e:case23} is no more than 
\begin{align*}
\sup_{0\leq\lambda \leq 1}  
\Bigl\lvert 
\sum_{n=1}^N &\textup{Mod}_{\theta_n} T_\lambda^l ( \varphi * \textup{Mod}_{-\theta_n}f) 
\Bigr\rvert
\\
& \lesssim \sup_{0\leq\lambda \leq 1}  \Bigl\lvert  \sum_{n=1}^N \textup{Mod}_{\theta_n} \Phi_{k,\lambda,l} ( \varphi* \textup{Mod}_{-\theta_n}f) \Bigr\rvert  + \min\{ 2^l, 2^{-l}\}  Mf ,  
\end{align*}
where $M$  the maximal function of Bourgain \eqref{e:M}, and we have used the estimate \eqref{e:mu<}.  
In view of Theorem~\ref{t:BMax},   this decomposition reduces the proof of \eqref{e:case23} to 
an appropriate bound for $ \sup_{0\leq\lambda \leq 1}  \lvert  \Phi _{k,\lambda, l  }  ^{\boldsymbol \theta }f\rvert $, where 
\begin{equation}\label{e:Phi_lambda}
 \Phi_{k,\lambda ,l}  ^{\boldsymbol \theta } f:=
\sum_{n=1}^N \textup{Mod}_{\theta_n} \Phi_{k,\lambda,l} ( \varphi* \textup{Mod}_{-\theta_n}f) .
\end{equation}
Namely, it remains to show that the $ L ^2 $ norm of the supremum over $ \lambda $ of the expression above 
is bounded by an absolute constant.  

The main tool is the technical square function inequality of  Lemma~\ref{l:sig}. To be explicit, let us state this corollary to the proof of that Lemma.  

\begin{cor}\label{c:sig} Under the assumptions of Lemma~\ref{l:sig}, there holds for all integers $ N$, and 
  $ 1$-separated points $ \{\theta _n \;:\; 1\leq n \leq N\}$ 
\begin{equation*}
\Bigl\lVert 
\sup _{0< \lambda < 1} 
\Bigl\lvert \sum_{n=1}^N \textup{Mod}_{\theta_n}  \mathcal F ^{-1} ( \sigma (\lambda , \xi )  
\widehat \varphi (\xi ) 
\widehat  f (\xi - \theta _n))  \Bigr\rvert\, 
\Bigr\rVert_2 \lesssim \lVert f\rVert_2. 
\end{equation*}
\end{cor}

Namely, the proof of Lemma~\ref{l:sig} depends only on some Plancherel calculations, and so extends to the 
setting above, due to the presence of the Fourier cut off function $ \widehat \varphi  $.  

We see that these two cases will follow if we show that the Fourier transform of 
$ \phi _{k, \lambda , l}$ satisfies the hypotheses of Lemma~\ref{l:sig}. 
For $  l \leq 0$, it does with $ \rho =1$, and for $ l >0$, it does 
with $ \rho =  2 ^{l/2}$,    since $ \lambda \approx 2 ^{-2k + l}$.  
By inspection, this is the consequence of the next two lemmas.

\begin{lemma}\label{EST1}
For $l \leq 0$, there holds 
\begin{equation}\label{e:EST1}
|\widehat{ \phi_{k,\lambda,l} }(\xi)| \lesssim \min\{ 2^k|\xi| , (2^k |\xi| )^{-1} \}.
\end{equation}
For $l \geq 0$,  there holds for $ 0< c < 1 < C < \infty $
 \begin{equation}\label{e:EST2}
 |\widehat{ \phi_{k,\lambda,l} }(\xi)| 
\lesssim \begin{cases} 2^{k-2l}|\xi| &\mbox{if } |\xi| \leq 2^{l-k} \\ 
2^{-l/2} &\mbox{if }  c 2 ^{l-k} \leq  |\xi| \leq C 2^{l-k} \\
(2^{k}|\xi|)^{-1} & |\xi| \ge C 2^{l-k}. \end{cases}  
\end{equation}
\end{lemma}

\begin{proof}
 For the case of $ l <0$,   recall from  \eqref{e:kl} that $ k _{l} =k$. 
 Since $ \widehat{ \phi_{k,\lambda,l} }(0)=0$ by construction, we estimate the derivative 
 in $ \xi $.   The derivative is 
\begin{equation} \label{e:Dsmall}
\bigl\lvert \frac d {d \xi }  \widehat{ \phi_{k,\lambda,l} }(\xi) \bigr\rvert 
= \Bigl\lvert  \int e (\lambda t ^2 - \xi t) t \{ \psi _{k} (t) 
- \mu (k,\lambda ,l)  \varphi_{k}(t)\}   \; dt    \Bigr\rvert \lesssim 2 ^{k}. 
\end{equation}
The first half of the bound \eqref{e:EST1} follows.  

The second half which applies when $ \lvert  \xi \rvert \geq C2 ^{k} $.  
Then, by \eqref{e:mu<}, 
\begin{equation*}
\bigl\lvert  \mu (k,\lambda ,l)   \widehat {  \varphi_{k_l}} \bigr\rvert \lesssim  2 ^{l}(2 ^{k} \lvert  \xi \rvert ) ^{-1} . 
\end{equation*}
This is better than what is claimed. On the other hand, in the integral 
\begin{equation*}
 \int e (\lambda t ^2 - \xi t ) \psi _{k} (t) \; dt 
\end{equation*}
the derivative of the phase function $ \lambda t ^2 - \xi t$ is at least $ c \lvert  \xi \rvert $ on the support of $ \psi $, 
so the bound as claimed follows in this case.

\medskip 

For the case of $ l \geq 0$, we are in the crucial oscillatory case.    Recall from \eqref{e:kl} that $ k_l = k - l$. 
Apply the second van der Corput test to the integral 
\begin{equation*}
\int e (\lambda t ^2 - \xi t) \psi _k (t) \;dt. 
\end{equation*}
It follows that the integral is bounded by 
\begin{equation*}
\left( \lVert \psi_k \rVert _{\infty } + \| \psi_k' \|_1 \right) \lambda ^{-1/2} \approx 2 ^{- l/2}. 
\end{equation*}
In particular, note that the  phase function $ \lambda t ^2 - \xi t$ has a critical point at $ t = 2 \lambda / \xi $.  
This critical point will be close to the support of $ \psi _{k}$ provided  $ \lvert  \xi \rvert  \simeq  2 ^{l-k}$.   
But,  $ \phi _{\lambda , l}$ is the difference of two terms.  In the second term, 
we have been careful to adjust the scale of $ \varphi _{k_l}$ in \eqref{e:phikl}, so by the 
estimate \eqref{e:mu<},  the case of $ c 2 ^{l-k} \leq \lvert  \xi \rvert < C 2 ^{l-k} $ in \eqref{e:EST2} follows.

If $ \lvert  \xi \rvert < c  2 ^{l-k}$, the dominant factor is that we have chosen 
$ \phi _{\lambda ,l}$ to have zero Fourier mode equal to $ 0$.  We then argue that the 
derivative is small, as we did in \eqref{e:Dsmall}. There are two terms in the derivative. 
In the first place, 
\begin{equation*}
\Bigl\lvert 
\frac d {d \xi } \int e (\lambda t ^2 - \xi t ) \psi _{k} (t) \; dt 
\Bigr\rvert  \lesssim 2 ^{k-2l}. 
\end{equation*}
Indeed, if $0< c<1$ is sufficiently small, then the derivative of the phase function $ \lambda t ^2 - \xi t$ 
is at least $ c 2 ^{-k+l} $ on the support of $ \psi _k (t)$. So the inequality follows by two integration by parts. 
And,  since $ k_l = k - l $, and $ \varphi $ has constant Fourier transform in a neighborhood of the  origin by \eqref{e:1phi}, 
the same derivative estimate holds for $ \mu (k, \lambda ,l) \widehat {\varphi _{k_l} } (\xi ) $.  
The case of $ \lvert  \xi \rvert < c 2 ^{l-k} $ follows. 

The last case in \eqref{e:EST2} is $ \lvert  \xi \rvert > C 2 ^{l-k} $. 
Observe that the estimate holds for $ \mu (k, \lambda ,l) \widehat {\varphi _{k_l} } (\xi ) $,  
\eqref{e:1phi}.  The other term is 
\begin{equation*}
 \int e (\lambda t ^2 - \xi t ) \psi _{k} (t) \; dt  . 
\end{equation*}
But, we have the lower bound on the derivative of the phase function of $ c \lvert  \xi \rvert $ 
on the support of $ \psi _k$, so the bound follows. 
\end{proof}

We have the following observation concerning the $\lambda$-derivative of $\phi_{k,\lambda,l}$:

\begin{lemma}\label{EST2}
The  derivatives
\[ 2^{-2k} \partial_{\lambda} \widehat{ \phi_{k,\lambda,l} }(\xi) , \qquad  2 ^{-2k+ l} < \lambda  < 2 ^{-2k+1+1}, 
\]
satisfy the   estimates  of Lemma~\ref{EST1}.  
\end{lemma}
 
\begin{proof}
By inspection, the derivative above  is an expression much like that of $ \widehat  {\phi _{k,\lambda ,l}} $, but with the function $ \psi _k  (t)$ 
replaced by $ (t2 ^{-k}) ^2 \psi (t 2 ^{-k})$.  But, this is an odd, smooth function with compact support,
which are the only properties  that were used to derive Lemma~\ref{EST1}.  The Lemma follows. 
\end{proof}

 \subsection{Case Four: $l > 10 \log_2 N$}
 The term to bound is \eqref{e:4}, which follows from summation over $ l > 10 \log_2 N$ 
 of the following inequality.  
 
 \begin{lemma}\label{l:4} For any $ l \geq 0$, there holds 
 \begin{equation}\label{e:4<}
 \bigl\| \sup_{0\leq\lambda \leq 1}  |T_\lambda^l f| \bigr\|_{2} \lesssim 2^{-l/4} \| f\|_{2}.  
 \end{equation}
 \end{lemma}
 
The proof of  Stein-Wainger \cite[Thm 1]{MR1879821} contains this result, with the 
right hand side replaced by $ 2 ^{- l \delta }$, for some $ \delta >0$.  
In this case, one can take $ \delta = 1/4$, but we omit the proof.

\section{Some Approximations} \label{s:circle}
Number-theoretic considerations dominate this section, namely a quantitative decomposition of the 
exponential sum given in \eqref{e:Mlb}. 
This decomposition is given in these terms.  
Let 
\[ \mathbf 1_{[-1/10,1/10]} \leq \chi \leq \mathbf 1_{[-1/5,1/5]} \]
be a smooth even bump function. The rationals in the two torus are the union over $ s\in \mathbb N $ of the collections 
\begin{equation}\label{e:Rs}
\R_s := \{ (A/Q,B/Q) \in \TT^2 : (A,B,Q) = 1, 2^{s-1 }\leq Q < 2^{s} \}.
\end{equation}
For each $s , j \in \mathbb N $, define the multiplier
\begin{align}\label{e:Ljs}
 L_{j,s}(\lambda,\beta)& := \sum_{ (A/Q,B/Q)\in  \R_s} S(A/Q,B/Q) H_j(\lambda - A/Q, \beta - B/Q) 
\\ & \qquad \qquad \qquad  \times \chi_s(\lambda- A/Q) \chi_s(\beta - B/Q), 
 \\ 
  \noalign{\noindent where  $\chi_s(t) := \chi(10^s t)$;  a continuous analogue of the sum is given by }   \label{e:Hj}
   &H_j(x,y) := \int e(x t^2 - y t) \psi_j(t) \; dt  ; 
   \\ \noalign{\noindent and a complete Gauss sum is given by }
    \label{e:gauss} 
    &S(A/Q,B/Q) := \frac{1}{Q} \sum_{r=0}^{Q-1} e(A/Q \cdot r^2 - B/Q \cdot r). 
\end{align}

Throughout this section, $ 0 < \epsilon \leq \tfrac 14$ denotes a fixed sufficiently small constant, 
that we will not attempt to optimize.  Then, define  $ L (\lambda , \beta ) = \sum_{j=1} ^{\infty } L _{j} (\lambda , \beta )$, 
where 
\begin{equation}  \label{e:Lj}
L _{j} (\lambda , \beta ) = \sum_{s \;:\;  1\leq s \leq  \epsilon j }  L _{j,s} (\lambda , \beta ).  
\end{equation}
The approximation theorem is as below. It is as expected, except for the derivative information 
on $ E_j (\lambda , \beta )$.   This additional information is easy to derive, and essential in the next section, 
which is why it is highlighted here. 

\begin{theorem}\label{t:smooth}  There  are choices of $ 0< \epsilon , \delta < \tfrac 14$ so that we can write 
$ M (\lambda , \beta ) = L (\lambda , \beta ) + E (\lambda , \beta )$, where $ L (\lambda , \beta )$ is defined 
in \eqref{e:Lj}, and the term $ E (\lambda , \beta )$ satisfies 
\begin{gather}\label{e:E1}
 E (\lambda , \beta ) = \sum_{j=1} ^{\infty } E _{j} (\lambda , \beta ) ,  
 \\  \label{e:E<}
   \lvert  E _j (\lambda , \beta )\rvert \lesssim 2 ^{- \delta  j},  
 \\ \label{e:dE<}
 \Bigl\lvert  \frac d {d \lambda } E _{j} (\lambda , \beta ) \Bigr\rvert \lesssim 2 ^{2j} . 
\end{gather}
The last two estimates are uniform over $ (\lambda , \beta ) \in \mathbb T ^2 $. 
\end{theorem}

The estimates on $ E_j$  match the hypotheses of Lemma~\ref{SOB}.
We turn to the proof, which is a standard application of the Hardy-Littlewood method in exponential sums. 
A core definition in this method is that of \emph{major boxes.} 

\begin{definition}  \label{d:major}
For $ (A/Q,B/Q)\in  \mathcal R_s$, where $ s\leq j \epsilon $, define the \emph{$ j$th major box at $ (A/Q, B/Q)$} to be 
the rectangle in $ \mathbb T ^2 $ given by 
\begin{equation}\label{e:major}
\mathfrak{M}_j(A/Q,B/Q) := \left\{ (\lambda, \beta) \in \mathbb{T}^2 : | \lambda - A/Q| \leq 2^{(\epsilon - 2)j},
| \beta - B/Q| \leq 2^{(\epsilon - 1)j} \right\}.
\end{equation}
\end{definition}

We collect the major boxes 
\begin{equation}\label{e:Major}
 \mathfrak{M}_j := \bigcup_{(A,B,Q) = 1 : Q \leq 2^{6j\epsilon} } \mathfrak{M}_j(A/Q,B/Q). 
\end{equation}
The $ 6\epsilon $ above plays a role in the proof of \eqref{e:minor}, which is a standard fact in the 
Hardy-Littlewood Circle method.

The union above is over  disjoint sets:  if $(\lambda,\beta) \in \mathfrak{M}_j(A/Q,B/Q) \cap \mathfrak{M}_j(A'/Q',B'/Q')$,  and 
$ A/Q \neq A'/Q'$, then 
\[ 2 \cdot 2^{(\epsilon-2)j} \geq |A/Q - \lambda| + |A'/Q' - \lambda| \geq |A/Q - A'/Q'| \geq \frac{1}{2^{6j\epsilon}},\]
which is a contradiction for $ 0 < \epsilon < 1/ 7 $.   If $ A/Q=A'/Q'$, then necessarily  necessarily 
$B/Q \neq B'/Q'$ and the same argument applies. 

Set 
\begin{equation} \label{e:Mjdef} 
 M_j(\lambda,\beta) := \sum_{m} e(\lambda m^2 - \beta m ) \psi_j(m).  
\end{equation}
Above, $ \psi _j$ is as in \eqref{e:1/t}.  

On any fixed major box, we have this approximation of $ M _{j} (\lambda , \beta )$, which is at the 
core of the proof of Theorem~\ref{t:smooth}.

\begin{lemma}\label{MAJor}  For $ 1\leq s \leq \epsilon j$,  $(A/Q,B/Q)\in  \mathcal R_s$,  and 
$ (\lambda , \beta ) \in \mathfrak M _{j} (A/Q, B/Q)$,  we have the approximation 
\begin{equation}\label{e:smooth}
 M_j(\lambda, \beta) = S(A/Q,B/Q) H_j(\lambda - A/Q, \beta - B/Q) + O(2^{(3\epsilon -1 )j }).
\end{equation}
The terms  above are defined in \eqref{e:Mjdef}, \eqref{e:gauss}, and \eqref{e:Hj}, respectively.  
\end{lemma}

\begin{proof}
Throughout the proof we write 
\begin{equation}\label{e:etas}
 \lambda = A/Q + \eta_2, \ \beta = B/Q + \eta_1,
\end{equation}
where $|\eta_2| \leq 2^{(\epsilon-2)j}$, and $|\eta_1|\leq 2^{(\epsilon-1)j}$.

The sum $ M _{j} (\lambda , \beta )$  is over integers, positive and negative,  in the support of $ \psi _j $.  
We  consider the sum over positive $m$, and decompose  into residue classes $ \operatorname {mod} Q$. 
Thus write $ m = pQ+r$, where $0 \leq r < Q \leq 2^{j\epsilon}$, 
and the integers $ p$ take values in an interval $ [c,d]$, where $ c =d/ 8   \approx 2 ^{j (1- \epsilon )}$, 
 in order to cover the support of $ \psi _j$. 

The argument of the exponential in \eqref{e:Mjdef} is, after reductions modulo 1,  
\[ \aligned
\lambda m^2 - \beta m &= (A/Q + \eta_2)(p Q + r )^2 - (B/Q + \eta_1)(pQ + r) \\
&\equiv r^2 A/Q - rB/Q + (pQ)^2 \eta_2 - pQ \eta_1 + O(2^{j(2\epsilon-1)}). \endaligned\]
That is, we can write 
\begin{equation*}
e (\lambda m ^2 - \beta m) = 
e (r^2 A/Q - rB/Q + (pQ)^2 \eta_2 - pQ\eta_1 ) + O(2^{j(2\epsilon-1)}).  
\end{equation*}

Then, we can write the sum $ \sum _{m \geq 0}
e(\lambda m^2 - \beta m) \psi_j(m)$ as follows.
\begin{align}
 &
\sum_{p \in I} \sum_{r=0}^{Q-1} 
e((r^2 A/Q - rB/Q + (pQ)^2 \eta_2 - pQ \eta_1 ) \psi_j(pQ + r)  + O(2^{(2\epsilon -1)j}) \\
&=   \sum_{r=0}^{Q-1}  e(r^2 A/Q - r B/Q)  \times \sum_{p \in I} 
e( \eta_2 (pQ)^2 - \eta_1pQ)\psi_j(pQ)  + O(2^{(2\epsilon -1)j}),
\\  \label{e:Ssum}
&= S (A/Q, B/Q) \times Q\sum_{p \in I} 
e( \eta_2 (pQ)^2 -\eta_1pQ)\psi_j(pQ)    + O(2^{(2\epsilon -1)j}). 
\end{align}
Above, we have appealed to several elementary steps. 
One of these is that $ \sum_{j} \lvert  \psi _j (m)\rvert \lesssim 1 $. 
Some additional terms in $r$ have been added,  so that the sum over $ p$ and $ r$ are over independent sets. 
These additions are absorbed into the Big-$O $ term.  
The argument of $ \psi _j $ is changed from $ pQ+r$ to $ pQ$, in view of the fact that 
the derivative of $ \psi _j$ is at most $ 2 ^{-2j}$, with the  change also being  absorbed into the Big-$O$ term.  
Finally, we appeal to the definition of the Gauss sum in \eqref{e:gauss} in order 
to have $ S (A/Q, B/Q) $ appear in the last line.

Comparing \eqref{e:Ssum} to the desired conclusion \eqref{e:smooth}, we show that 
\begin{equation}\label{e:SSum}
Q\sum_{p \in I}  e( \eta_2 (pQ)^2 - \eta_1 pQ)  \psi_j(pQ)   = \int_{0}^\infty e(\eta_2 t^2  -\eta_1 t) \psi_j(t) \; dt + O(2^{(2\epsilon - 1)j}). 
\end{equation}
The same argument to this point will apply to the sum over negative $ m$, so that our proof will then be complete.

But the proof of \eqref{e:SSum} is straight forward.  For fixed $ p\in I$, and $ 0\leq t \leq Q$, we have 
\begin{align*}
\bigl\lvert 
 e( \eta_2 (pQ)^2  &- \eta_1 pQ)   \psi_j(pQ) -  e( \eta_2 (pQ+t)^2 - \eta_1 (pQ+t))  \psi_j(pQ+t) 
\bigr\rvert
\\& \lesssim  
\lvert    e (\eta_2 (pQ)^2) - e (\eta_2 (pQ+t)^2)\rvert 2 ^{-j} 
+ \lvert    e (-\eta _1 pQ ) - e (-\eta _1 (pQ+t))\rvert  2 ^{-j} 
\\& \quad +  
\lvert  \psi_j(pQ)  -  \psi_j(pQ+t)  \rvert.  
\end{align*}
Each of the three terms on the right is at most $ O (2 ^{(2 \epsilon -2)j})$. 
In view of the fact that there are $ Q \lvert  I\rvert \lesssim 2 ^{j} $ summands on the left in 
\eqref{e:SSum}, this is  all that we need to conclude the inequality in \eqref{e:SSum}. 
The three terms on the right above are bounded in reverse order. 
    Since the derivative of $ \psi _j$ is at most $ 2 ^{-2j}$, 
\begin{equation*}
\lvert  \psi_j(pQ)  -  \psi_j(pQ+t)  \rvert \lesssim 2 ^{-2j}   t  \lesssim 2 ^{ (\epsilon -2)j},  
\end{equation*}
since $0\leq  t\leq Q \leq 2 ^{\epsilon j}$.  Recalling that $ \lvert  \eta _1\rvert \leq 2 ^{ (\epsilon -1)j} $, there holds 
\begin{equation*}
 \lvert    e (-\eta _1 pQ ) - e (-\eta _1 (pQ+t))\rvert  2 ^{-j} \leq 
 \lvert  \eta _1 \rvert t \leq   2 ^{ (2 \epsilon -2)j}.  
\end{equation*}
Recalling that $ \lvert  \eta _2\rvert\leq 2 ^{ (\epsilon -2)j} $, there holds 
\begin{equation*}
\lvert    e (\eta_2 (pQ)^2) - e (\eta_2 (pQ+t)^2)\rvert 2 ^{-j} \lesssim 
\lvert  \eta _2 \rvert p Q t 2 ^{-j} \lesssim 2 ^{ (2 \epsilon -2j)}.  
\end{equation*}
Thus, \eqref{e:SSum} holds.  

\end{proof}

These  important estimates are familiar.  

\begin{lemma}\label{l:three}  These estimates hold. 
\begin{enumerate}

\item  (Complete Gauss Sums) For any $ 0 < \nu < \frac 12 $, 
uniformly in $ A, B, Q$, 
\begin{equation}\label{e:gauss<} 
\lvert   S(A/Q,B/Q) \rvert \lesssim_\nu Q^{-\nu}, \qquad  (A,B,Q) =1. 
\end{equation}

\item  (Minor Arcs)
There is a $ \delta = \delta (\epsilon )$ so that uniformly in $ j\geq 1$,   
\begin{equation}\label{e:minor}
\lvert  M_j(\lambda,\beta)\rvert  \lesssim 2 ^{- \delta j   } , \qquad (\lambda ,\beta )\not\in \mathfrak M_j. 
\end{equation}

\item  (Oscillatory Estimate) 
Define $ \|(x,y)\|_j := 2^{2j}|x| + 2^j|y|$. The integral $ H_j$ defined in \eqref{e:Hj} satisfies 
\begin{equation}\label{e:Hj<}
 |H_j(x,y)| \lesssim \min \{ \| (x,y) \|_j , \|(x,y) \|_j^{-1/2} \}. 
\end{equation}

\end{enumerate}

\end{lemma}

\begin{proof}[Sketch of Proof]
The first estimate is a fundamental one on Gauss sums, see Hua \cite[\S 7, Theorem 10.1]{MR665428}. 

\medskip 
The second estimate can be deduced from   Bourgain's \cite[Lemma 5.6]{MR1019960}. 
Estimates of the latter type are well known, the subject of the Hardy-Littlewood method. 

\medskip 

The third estimate is simple oscillatory estimate.   It is clear that $ \lvert  H_j (x,y)\rvert $ is 
always bounded by a constant.   Assume that $ \|(x,y)\|_j $ is small. Then, since $ H (0,0)=0$, 
\begin{equation*}
 |H_j(x,y)| \lesssim \|(x,y)\|_j  \int  (2 ^{-2j}t ^2  + 2 ^{j} \lvert  t\rvert ) \lvert  \psi _{j} (t)\rvert \;dt 
 \lesssim \|(x,y)\|_j. 
\end{equation*}
If $ \|(x,y)\|_j $ is large, then the phase function has derivative $ 2 x t + y$, which has a critical 
point at $ t=- y/2x$. 
From the second van der Corput estimate, we have 
\begin{equation*}
 |H_j(x,y)| \lesssim  \left( \lVert \psi _j\rVert _{L^\infty } + \| \psi_j' \|_{L^1} \right) \lvert  x\rvert ^{-1/2} \lesssim [ 2 ^{2j} \lvert  x\rvert  ] ^{-1/2}.  
\end{equation*}
Provided $ 2 ^{2j} \lvert  x\rvert \geq c 2 ^{j} \lvert  y\rvert  $, that completes the proof of \eqref{e:Hj<}. 
If this fails, it follows that the phase function has derivative at least $c \lvert  y\rvert $ on the support 
of $ \psi _j$, so that the estimate follows from an integration by parts. 

\end{proof}

This Lemma is elementary in nature.  

\begin{lemma}\label{l:Lj<}  For $ \epsilon  > 0$ sufficiently small, and for all integers $ j$, 
\begin{equation}\label{e:Lj<}  
\lvert   L _{j} (\lambda , \beta )\rvert \lesssim 2 ^{ -  \epsilon j/3} \qquad 
(\lambda , \beta ) \not\in  \mathfrak M _{j} . 
\end{equation}
And, for $ 1\leq s \leq \epsilon j$, there holds 
\begin{equation}\label{e:LJ<}  
\sum_{ 1 \leq s' \neq s \leq \epsilon j}\lvert   L _{j,s'} (\lambda , \beta )\rvert \lesssim 2 ^{ -  \epsilon j/3} \qquad 
(\lambda , \beta ) \in \bigcup _{ (A/Q,B/Q)\in  \mathcal R_s}  \mathfrak M _{j} (A,B,Q).  
\end{equation}
\end{lemma}

\begin{proof}
We show this inequality, with a better exponent.   For $ 1\leq s \leq \epsilon j$, there holds 
\begin{equation}\label{e:Ljs<}  
\lvert   L _{j,s} (\lambda , \beta )\rvert \lesssim 2 ^{ -  \epsilon j/2} \qquad 
(\lambda , \beta ) \not\in \bigcup _{ (A/Q,B/Q)\in  \mathcal R_s}  \mathfrak M _{j} (A,B,Q).  
\end{equation}
One should recall the definition of the major box in \eqref{e:major},  as well as the definition of $ L _{j,s}$ in \eqref{e:Ljs}. 
As, $ 0< \epsilon <1$ is small,   and $ 0 < s < \epsilon j$, 
we should consider the case of $  100 ^{-s+1} \geq  2 ^{(\epsilon -1) j}$. 

There are two factors that come into play. 
The first is that the cut-off functions that enter into the definition of $ L _{j,s}$ are disjointly supported.  
Namely, for $ (A, B,Q) \neq (A',B',Q') \in \mathcal R_s$, there holds 
\begin{equation*}
\chi_s(\lambda- A/Q) \chi_s(\beta - B/Q) \cdot \chi_s(\lambda- A'/Q') \chi_s(\beta - B'/Q') \equiv  0. 
\end{equation*}
We omit this argument, as we have already used a variant in discussing the disjoint union in \eqref{e:Major}. 
The second is the decay estimate \eqref{e:Hj<}.  It implies in particular that 
provided $ (\lambda , \beta ) \not\in \mathfrak M _{j} (A,B,Q)$, 
\begin{align*}
\lvert  H _{j} (\lambda , \beta )\rvert  &\chi_s(\lambda- A/Q) \chi_s(\beta - B/Q) 
\\&
\lesssim   [ 2 ^{2j} \lvert  \lambda -A/Q\rvert + 2 ^{j} \lvert  \beta -B/Q\rvert] ^{-1/2}   
\lesssim  2 ^{- \epsilon j /2}, 
\end{align*}
The proof  of \eqref{e:Ljs<} is complete. 

\smallskip 

To see that \eqref{e:Lj<} holds, from the definition in \eqref{e:Lj}, we have 
\begin{equation*}
\lvert   L _{j} (\lambda , \beta )\rvert \leq 
\sum_{ 1 \leq s \leq \epsilon j}\lvert   L _{j,s} (\lambda , \beta )\rvert , 
\end{equation*}
and so we can use the estimate \eqref{e:Ljs<} at most $ \epsilon j$ times.  The same argument applies to 
\eqref{e:LJ<}.  

\end{proof}

\begin{proof}[Proof of Theorem~\ref{t:smooth}] 
Recalling the definition of $ M_j$ in \eqref{e:Mjdef} and $ L_j$ in \eqref{e:Lj}, we should check that 
\begin{equation*}
E_j (\lambda , \beta ) := M_j (\lambda , \beta ) - L _{j} (\lambda , \beta ) 
\end{equation*}
satisfy the conclusions of the Lemma, namely the size condition \eqref{e:E<} and the derivative condition 
\eqref{e:dE<}.  
For $ (\lambda , \beta ) \not\in \mathfrak M_j$,  combine the two inequalities \eqref{e:minor} and \eqref{e:Lj<} to see that 
\begin{equation*}
\lvert  E_j (\lambda , \beta )  \rvert   \lesssim 2 ^{- \eta  j}  
\end{equation*}
where $ \eta = \eta (\epsilon ) = \min \{\delta (\epsilon ), \epsilon /3\}$.  

Therefore, it remains to consider the case of $ (\lambda , \beta ) \in \mathfrak M_j $, 
whence $ (\lambda , \beta ) \in \mathfrak M_j (A/Q, B/Q) $ for a unique $ (A/Q,B/Q)\in  \mathcal R_s$, 
and $ 1\leq s \leq 3\epsilon j$.  
This condition means that $ \lvert  \lambda -A/Q\rvert \leq  2 ^{(\epsilon -2)j} $ and $ \lvert  \beta - B/Q\rvert \leq  
2 ^{ (\epsilon -1) j}$.  By \eqref{e:smooth}, we have 
\begin{equation}
M_j(\lambda, \beta) = S(A/Q,B/Q) H_j(\lambda - A/Q, \beta - B/Q) + O(2^{(3\epsilon -1 )j })
\end{equation}
The first term on the right equals $ L _{j,s} (\lambda , \beta )$ if  there holds 
\begin{equation}  \label{e:10}
 \lvert  \lambda -A/Q\rvert \leq  10 ^{-s+1} , \quad \lvert  \beta - B/Q\rvert \leq  
10 ^{ -s+1}
\end{equation}
by inspection of the definition of $ L _{j,s}$ in \eqref{e:Ljs}.   
Now, note that  since we only know $ 1\leq s \leq \epsilon j $, and we can take $ 0 < \epsilon <1$ small, 
these conditions are weaker than $  (\lambda , \beta ) \in \mathfrak M_j (A/Q, B/Q) $, and so 
\begin{equation*}
 S(A/Q,B/Q) H_j(\lambda - A/Q, \beta - B/Q) = L _{j,s} (\lambda , \beta ).  
\end{equation*}
It remains to account for the $ L _{j,s'} (\lambda , \beta )$, where $ s'\neq s$. But this is the inequality \eqref{e:LJ<}.

To conclude, the derivative condition is entirely elementary.  
The derivative with respect to $ \lambda $  of $M_j (\lambda , \beta )  $ is clearly dominated by $ 2 ^{2j}$, 
while that of $ L_j$ is dominated by $ 2 ^{2j}+10^{2s}$, where $ 1\leq s \leq \epsilon j$.  So it is dominated 
by $ 2 ^{2j}$, provided $ 0<\epsilon $ is small enough, as we assume.

\end{proof}

\section{Completing the Proof of Theorem~\ref{t:main}}

We combine the results of the previous sections to prove Theorem~\ref{t:main}. 
Our task is to bound 
\begin{equation*}
\Bigl\lVert 
 \sup_{\lambda \in \Lambda} \bigl\lvert  \F^{-1}_{\beta}  \bigl(M (\lambda , \beta ) \hat{f}(\beta) \bigr)\bigr\rvert \, 
\Bigr\rVert_2 
\end{equation*}
where $ M (\lambda , \beta )$ in \eqref{e:Mlb} is viewed as a Fourier multiplier 
on $ \mathbb R $, and supported on $ 0< \beta <1 $, say. 
Recall Theorem~\ref{t:smooth}, so that we have a decomposition 
\begin{equation*}
M (\lambda , \beta ) = L (\lambda , \beta ) + E (\lambda , \beta ) 
\end{equation*}
meeting the conditions involved in that Theorem.  
Moreover, $ \Lambda $ satisfies our arithmetic Minkowski dimension condition, 
as specified in Definition~\ref{d:amk}.  
We require that the  dimension be at most $ 0< d _{\Lambda } = d _{\epsilon,  \delta  } <1 $, 
where $ 0< \epsilon, \delta < 1 $ is as in Theorem~\ref{t:smooth}.  

The dimension will enter in two different ways. The first is the elimination of the `error terms' $ E_j (\lambda , \beta )$ 
in \eqref{e:E1}.  This step only requires a Minkowski dimension assumption. 
Recall the quantitative estimates \eqref{e:E<},  \eqref{e:dE<} on the $ E_j$. Combine them with 
Lemma~\ref{SOB}.  In particular, the assumption \eqref{e:SOB} holds with $ a \approx 2 ^{- \delta j}$, 
and $ A \approx 2 ^{2j}$.    Assuming that the Minkowski dimension of $ \Lambda $ is at most $ d _{\Lambda } \leq \frac \delta2  $, 
we can apply \eqref{e:tau} to conclude that 
\begin{equation*}
\Bigl\lVert 
 \sup_{\lambda \in \Lambda} \bigl\lvert \bigl(
  \F^{-1}_{\beta}  E_j (\lambda , \beta ) \hat{f}(\beta) 
 \bigr) \bigr\rvert \, 
\Bigr\rVert_2 \lesssim 2 ^{- \delta j/4} \lVert f\rVert_2 . 
\end{equation*}
This is summable in $ j\in \mathbb N $.  Hence it remains to consider $ L (\lambda , \beta )$.  

\medskip 

Recall that $ L (\lambda , \beta ) $ is defined in \eqref{e:Lj}.  That definition was formulated for the 
deduction of Theorem~\ref{t:smooth}, and our purposes now are a bit different. 
Instead of the representation in \eqref{e:Lj}, we write $ L (\lambda , \beta ) = \sum_{s\in \mathbb N } L ^{s} (\lambda , \beta )$, 
with 
\begin{gather}\label{e:Ls}
\begin{split}
L ^{s} (\lambda , \beta ) :=
\sum_{ (A/Q,B/Q)\in  \R_s} S(A/Q,B/Q)  &H^s(\lambda-A/Q,\beta-B/Q)
\\ &\times \chi _s(\lambda - A/Q) \chi _s(\beta - B/Q), 
\end{split}
\end{gather}
where $ \mathcal R_s$ is as in \eqref{e:Rs},  $ S (A/Q, B/Q)$ is the Gauss sum as in \eqref{e:gauss},   
\begin{equation*}
H ^{s} (\lambda , \beta ) = \sum_{j \;:\; s \leq \epsilon j} H _{j} (\lambda , \beta )
\end{equation*}
where $ 0< \epsilon <1$ is as in Theorem~\ref{t:smooth}, and $ H_j$ is defined in \eqref{e:Hj}, and 
the function $ \chi  _s$ is as in \eqref{e:Ljs}.  

As the estimate in this Lemma is summable in $ s   \geq 0$, it is the last step in the proof of our main theorem.  
\begin{lemma}\label{l:L}
There is a  $ 0 < \kappa < 1$, so that for  all integers $ s$, 
\begin{equation}\label{e:L<}
\Bigl\lVert 
 \sup_{\lambda \in \Lambda} \bigl\lvert \bigl(
  \F^{-1}_{\beta}  L^s (\lambda , \beta ) \hat{f}(\beta) 
 \bigr) \bigr\rvert \, 
\Bigr\rVert_2 \lesssim   2 ^{- \kappa  s} \lVert f\rVert_2 . 
\end{equation}
Above, the implied  constant depends only on  $ C _{\Lambda } $ is as in Definition~\ref{d:amk}.  
\end{lemma}

\begin{proof}
Cover $ \Lambda $ by intervals $ \Lambda _1 ,\dotsc, \Lambda _N$ of width $   10 ^{-s+1}$, 
each $ \Lambda _n$ centered at a rational  $ r_n$ with denominator at most $ C _{\Lambda } 10 ^{d (s-1)}$.  
As follows from Definition~\ref{d:amk}, the number of intervals $ N$ is at most $ C _{\Lambda } 10 ^{2ds}$.  
It suffices to prove that  uniformly in $ 1\leq n \leq N$, 
\begin{equation}\label{e:LL<}
\Bigl\lVert 
 \sup_{\lambda \in \Lambda_n} \bigl\lvert \bigl(
  \F^{-1}_{\beta}  L^s (\lambda , \beta ) \hat{f}(\beta) 
 \bigr) \bigr\rvert \, 
\Bigr\rVert_2 \lesssim  2 ^{- 2 s/7} \lVert f\rVert_2 . 
\end{equation}
If $ 0< d < 1$ is sufficiently small,  we can then just use the triangle inequality 
to complete the proof.

For $ 1\leq n \leq N$, define  
\begin{equation*}
\begin{split}
L_n  (\lambda , \beta ) 
 :=
\sum_{ \substack{B \;:\;  (A/Q,B/Q)\in  \R_s\\ A/Q =r_n }} S(r_n,B/Q) & H^s(\lambda-r_n,\beta-B/Q)
\\& \times 
\chi _s(\lambda - r_n) \chi _s(\beta - B/Q). 
\end{split}
\end{equation*}
Above, we have added the restriction that $ A/Q= r_n$.  The essential point is that 
in \eqref{e:LL<}, we can replace  $ L ^{s} (\lambda , \beta )$ by $ L_n (\lambda , \beta )$. 
And then, the inequality \eqref{e:LL<} is  a consequence of  the key maximal inequality of Theorem~\ref{KEY} and the Gauss sum 
estimate \eqref{e:gauss<}, as we now explain.

To be explicit, set 
\begin{equation*}
\widehat  g _n(\beta )=
\sum_{ \substack{B \;:\;  (A/Q,B/Q)\in  \R_s\\ A/Q =r_n }} 
S(r_n,B/Q)    \widehat f (\beta ) \cdot \mathbf 1_{  \lvert  \beta - B/Q \rvert < 10 ^{-s} } . 
\end{equation*}
From the Gauss sum estimate \eqref{e:gauss<},  we have $ \lVert g_n\rVert_2 \lesssim 2 ^{-3s/7} \lVert f\rVert_2$.  
But, we also have 
\begin{equation*}
 L_n  (\lambda , \beta ) \hat{f}(\beta)  = 
\sum_{ \substack{B \;:\;  (A/Q,B/Q)\in  \R_s\\ A/Q =r_n }}   H^s(\lambda-A/Q,\beta-B/Q) \widehat g_n (\beta ).  
\end{equation*}
Moreover, the term on the right above is an instance of the terms in the maximal inequality of 
Theorem~\ref{KEY}, with $ N \approx 2 ^{s}$, and the separation parameter $ \tau \approx 2 ^{-2s}$. 
Therefore, it is a consequence of that Theorem that we have 
\begin{equation*}
\Bigl\lVert 
\sup _{ \lambda \in \Lambda _n } 
\Bigl\lvert 
\sum_{ \substack{B \;:\;  (A/Q,B/Q)\in  \R_s\\ A/Q =r_n }}  H^s(\lambda-A/Q,\beta-B/Q) \widehat g_n (\beta ) 
\Bigr\rvert
\Bigr\rVert_2 \lesssim s ^{2}  \lVert g_n\rVert_2. 
\end{equation*}
This estimate is uniform in $ 1\leq n \leq N$, hence \eqref{e:LL<} follows. 

\end{proof}


\end{document}